\documentclass[11pt, a4paper, reqno]{amsart}

\pdfoutput=1

\usepackage{mathpazo, tgpagella}
\usepackage{tikz}
\usepackage{mathtools}
\usepackage{tikz-cd}
\tikzcdset{diagrams={nodes={inner sep=15pt}}}
\usepackage{enumerate}
\usepackage[margin=3.4cm]{geometry}
\usepackage{bm}
\usepackage{microtype}
\usepackage[foot]{amsaddr}

\newtheorem{theorem}{Theorem}
\newtheorem{corollary}[theorem]{Corollary}
\newtheorem{lemma}[theorem]{Lemma}
\newtheorem{proposition}[theorem]{Proposition}

\theoremstyle{definition}
\newtheorem{example}[theorem]{Example}
\newtheorem{remark}[theorem]{Remark}

\title[Completely Degenerate Equilibria of the Kuramoto Model on Networks]
	{\large Completely Degenerate Equilibria \\ of the Kuramoto Model on Networks}
\author{Davide Sclosa}
\address{Vrije Universiteit Amsterdam}
\email{\ttfamily d.sclosa@vu.nl}

\newcommand {\T} {\mathbb T}
\newcommand {\R} {\mathbb R}
\newcommand {\Z} {\mathbb Z}

\begin{document}

\begin{abstract}
Kuramoto Networks contain non-hyperbolic equilibria whose stability is sometimes difficult
to determine. We consider the extreme case in which all Jacobian
eigenvalues are zero.
In this case linearizing the system
at the equilibrium leads to a Jacobian matrix which is zero in every entry.
We call these equilibria completely degenerate.
We prove that they exist for certain intrinsic frequencies if and only if
the underlying graph is bipartite, and that they do not exist for generic intrinsic frequencies.
In the case of zero intrinsic frequencies, we prove that they exist if and only if
the graph has an Euler circuit such that
the number of steps between any two visits at the same vertex is a multiple of~$4$.
The simplest example is the cycle graph with~$4$ vertices.
We prove that graphs with this property exist for every number of vertices~$N\geq 6$
and that they become asymptotically rare for~$N$ large.
Regarding stability, we prove that for any choice of intrinsic frequencies, any coupling strength
and any graph with at least one edge,
completely degenerate equilibria are not Lyapunov stable.
As a corollary, we obtain that
stable equilibria in Kuramoto Networks must have at least one strictly negative eigenvalue.
\end{abstract}

\maketitle

\section{Motivation}
Kuramoto Networks are widespread in neuroscience, biology, chemistry and
engineering,
as a natural framework for modeling
synchronization and pattern formation~\cite{acebron2005kuramoto, arenas2008synchronization, dorfler2014synchronization, dorfler2013synchronization}.
Changing the underlying graph, even just removing one edge,
can lead to a dramatically different dynamical behavior.
Therefore, great effort has been put into understanding dynamics on different
classes of graphs:
complete graphs~\cite{jadbabaie2004stability, Taylor2012},
cycles~\cite{canale2009, Wiley2006, roy2012synchronized},
bipartite graphs~\cite{verwoerd2009computing},
planar graphs~\cite{delabays2017multistability},
dense graphs~\cite{Kassabov2021, Ling2019, lu2020, Taylor2012, Yoneda2021},
sparse graphs~\cite{sokolov2019sync}, $3$-regular graphs~\cite{DeVille2016},
trees~\cite{Dekker2013, Jafarian2018}, and stars~\cite{Chen2019}.
Here, rather than focusing on a particular type of graph,
we focus on a particular type of equilibrium.
We believe that degenerate equilibria in Kuramoto Networks
are always caused by some special combinatorial properties of the underlying graph,
which can be exploited to determine stability when linearization cannot.
The results in this paper support this belief.

Let~$G$ be a graph
with vertices~$1,\ldots,N$ and adjacency matrix~$(a_{jk})_{j,k}$.
To every vertex~$k$ we associated a~\emph{phase}~$\theta_k$
in the $1$-dimensional~$\T = \R/2\pi\Z$,
and a constant \emph{intrinsic frequency}~$\omega_k\in \R$.
The connection between vertices is modulated by a \emph{coupling strength}~$K>0$.
We call~\emph{Kuramoto Network on~$G$} the coupled dynamical system
\begin{equation} \label{eq:main_intro}
	\dot \theta_k = \omega_k + K\sum_{j=1}^N a_{jk} \sin(\theta_j-\theta_k),
	\qquad k=1,\ldots,N.
\end{equation}
The equilibria of~\eqref{eq:main_intro} are never isolated, due to the presence
of phase-shift symmetry, and in particular they are never hyperbolic.
Some equilibria are hyperbolic up to symmetry:
a typical example is the fully synchronized state, in which all the vertices
share the same phase.
In this paper we are interested in equilibria that are non-hyperbolic up to symmetry.

We say that an equilibrium is~\emph{completely degenerate} if the Jacobian matrix,
obtained by linearizing the vector field at the equilibrium, is zero in every entry.
As we will see, the Jacobian matrices obtained by linearizing~\eqref{eq:main_intro}
are symmetric. Therefore, they are zero in every entry if and only if all the
eigenvalues are zero.

Completely degenerate equilibria are interesting.
First, they are a source of counterexamples.
For instance, they show that
Taylor's inequalities~\cite[Lemma 2.1]{Taylor2012}
are not sufficient to guarantee stability.
Second, they are equilibria with critical edges
(see~\cite{DeVille2016} for the definition), for which little is known.
Third, they are closely related
to Eulerian graphs, a class of graphs not yet analyzed
in Kuramoto Networks literature.
Fourth, as their stability cannot be determined by linear stability analysis,
they require alternative techniques: in this paper
we will combine results on real-analytic gradient systems with graph-theoretical arguments.

\begin{figure}
\begin{tikzcd}
\begin{tikzpicture} [scale=0.7]
\draw (2,2) -- (-2,2);
\draw (-2,2) -- (-2,-2);
\draw (-2,-2) -- (2,-2);
\draw (2,-2) -- (2,2);
\fill [green] (2, 2) circle (0.15);
\fill [blue] (2, -2) circle (0.15);
\fill [red] (-2, 2) circle (0.15);
\fill [yellow] (-2, -2) circle (0.15);
\end{tikzpicture}
&
\begin{tikzpicture} [scale=0.7]
\draw (1,1) -- (-1,1);
\draw (-1,1) -- (-1,-1);
\draw (-1,-1) -- (1,-1);
\draw (1,-1) -- (1,1);
\draw (1+0.5,1+0.25) -- (-1+0.5,1+0.25);
\draw (-1+0.5,1+0.25) -- (-1+0.5,-1+0.25);
\draw (-1+0.5,-1+0.25) -- (1+0.5,-1+0.25);
\draw (1+0.5,-1+0.25) -- (1+0.5,1+0.25);
\draw (1,1) -- (1+0.5, 1+0.25);
\draw (-1,1) -- (-1+0.5, 1+0.25);
\draw (1,-1) -- (1+0.5, -1+0.25);
\draw (-1,-1) -- (-1+0.5, -1+0.25);
\draw (2.5, 2.5) -- (-2.5, 2.5);
\draw (-2.5, 2.5) -- (-2.5, -2.5);
\draw (-2.5, -2.5) -- (2.5, -2.5);
\draw (2.5, -2.5) -- (2.5, 2.5);
\draw (2.5+1, 2.5+ 0.5) -- (-2.5+1, 2.5+ 0.5);
\draw (-2.5+1, 2.5+ 0.5) -- (-2.5+1, -2.5+ 0.5);
\draw (-2.5+1, -2.5+ 0.5) -- (2.5+1, -2.5+ 0.5);
\draw (2.5+1, -2.5+ 0.5) -- (2.5+1, 2.5+ 0.5);
\draw (2.5, 2.5) -- (2.5+1, 2.5+0.5);
\draw (-2.5, 2.5) -- (-2.5+1, 2.5+0.5);
\draw (2.5, -2.5) -- (2.5+1, -2.5+0.5);
\draw (-2.5, -2.5) -- (-2.5+1, -2.5+0.5);
\draw (1,1) -- (2.5, 2.5);
\draw (-1,1) -- (-2.5, 2.5);
\draw (1,-1) -- (2.5, -2.5);
\draw (-1,-1) -- (-2.5, -2.5);
\draw (1+0.5,1+0.25) -- (2.5+1, 2.5+0.5);
\draw (-1+0.5,1+0.25) -- (-2.5+1, 2.5+0.5);
\draw (1+0.5,-1+0.25) -- (2.5+1, -2.5+0.5);
\draw (-1+0.5,-1+0.25) -- (-2.5+1, -2.5+0.5);
\fill [green] (1, 1) circle (0.15);
\fill [blue] (1, -1) circle (0.15);
\fill [red] (-1, 1) circle (0.15);
\fill [yellow] (-1, -1) circle (0.15);
\fill [red] (1 + 0.5, 1 + 0.25) circle (0.15);
\fill [yellow] (1 + 0.5, -1+ 0.25) circle (0.15);
\fill [green] (-1+ 0.5, 1+ 0.25) circle (0.15);
\fill [blue] (-1+ 0.5, -1+ 0.25) circle (0.15);
\fill [blue] (2.5, 2.5) circle (0.15);
\fill [green] (2.5, -2.5) circle (0.15);
\fill [yellow] (-2.5, 2.5) circle (0.15);
\fill [red] (-2.5, -2.5) circle (0.15);
\fill [yellow] (2.5 + 1, 2.5 + 0.5) circle (0.15);
\fill [red] (2.5 + 1, -2.5+ 0.5) circle (0.15);
\fill [blue] (-2.5+ 1, 2.5+ 0.5) circle (0.15);
\fill [green] (-2.5+ 1, -2.5+ 0.5) circle (0.15);
\end{tikzpicture}
\end{tikzcd}
\caption{Completely degenerate equilibria on the square and on the hypercube.
Blue, green, red and yellow
denote $0$, $\pi/2$, $\pi$ and~$3\pi/2$ respectively.}
\label{fig:hypercube}
\end{figure}
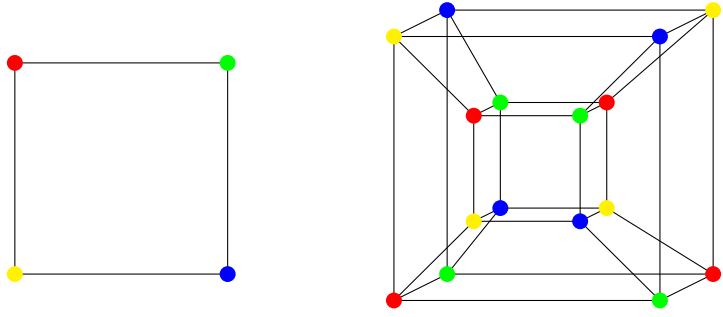

\section{Zero Intrinsic Frequencies}
\subsection{Classification}
Let us begin by analyzing completely degenerate equilibria in the case of zero intrinsic
frequencies, that is
$
	\omega_1=\cdots=\omega_N=0.
$
Up to rescaling time we can suppose~$K=1$, obtaining
\begin{equation} \label{eq:main}
	\dot \theta_k = \sum_{j=1}^N a_{jk} \sin(\theta_j-\theta_k),\qquad k=1,\ldots,N.
\end{equation}
Let~$F$ denote the vector field induced by~\eqref{eq:main}.
Let~$DF(\theta)$ denote the differential of the vector field at~$\theta$.
For every~$j,k\in \{1,\ldots,N\}$ we have
\begin{equation} \label{eq:DF}
	DF(\theta)_{jk} =
    	\begin{cases*}
    		a_{jk} \cos(\theta_j - \theta_k) & if $j\neq k,$ \\
    		- \sum_j a_{jk} \cos(\theta_j - \theta_k) & if $j=k$.
    	\end{cases*}
\end{equation}
If~\eqref{eq:main} contains a completely degenerate equilibrium,
we say that the graph~$G$~\emph{admits} completely degenerate equilibria.

\begin{lemma} \label{lem:criterion}
A point~$\theta\in \T^N$ is a completely degenerate equilibrium
if and only if, for every vertex~$k$, half of its neighbors have phase~$\theta_k + \pi/2$ and
the other half~$\theta_k - \pi/2$.
\end{lemma}
\begin{proof}
From~\eqref{eq:main} and~\eqref{eq:DF} it follows that a point~$\theta$
is a completely degenerate equilibrium if and only if
\begin{align}
&\sum_{j} a_{jk} \sin(\theta_j - \theta_k) = 0 \label{eq:sine_condition} \\
\intertext{for every vertex~$k$ and}
&\cos(\theta_j - \theta_k) = 0 \label{eq:cosine_condition}
\end{align}
for every edge~$jk$.
These equations are satisfied if and only if for every edge~$jk$
we have~$\sin(\theta_j - \theta_k) = \pm 1$ and for every vertex half
the sines are equal to~$+1$, the other half to~$-1$.
\end{proof}

Let us recall some definitions from graph theory.
A \emph{cycle} is a closed graph walk without repeated vertices.
A \emph{circuit} is a closed graph walk without repeated edges.
An \emph{Euler circuit} is a circuit that uses all the graph edges. If an Euler circuit exists,
the graph is called~\emph{Eulerian}. Notice that every cycle is a circuit.

The following theorem characterizes graphs
admitting completely degenerate equilibria
as well as the set of completely degenerate equilibria given a graph.
Notice that, since distinct connected components have independent dynamics,
it is enough to consider connected graphs.

\begin{theorem} \label{thm:1}
A connected graph~$G$ admits a completely degenerate equilibrium if and only if
there is an Euler circuit such that
the number of steps between any two visits at the same vertex is a multiple of~$4$.
Every completely degenerate equilibrium is obtained by fixing such an Eulerian circuit,
choosing the phase of one vertex arbitrarily and increasing the phase of the next vertex
by~$\pi/2$ at each step of the circuit.
\end{theorem}
\begin{proof}
Suppose that~$G$ admits completely degenerate equilibria.
Fix a completely degenerate equilibrium~$\theta$.
We say that a circuit has the property~$P$ if at each step the phase increases by~$\pi/2$.
By Lemma~\ref{lem:criterion} it follows that every edge of~$G$
is contained in a circuit satisfying~$P$.
Moreover, it follows that
if we remove from~$G$ the edges of a circuit satisfying~$P$, then~$\theta$
is still a completely degenerate equilibrium.
Therefore, the set of edges of~$G$ is the union of edge-disjoint circuits satisfying~$P$.
If two such circuits intersect in a vertex, their union can be walked in a way
that makes it a circuit satisfying~$P$.
Since~$G$ is connected, we conclude that there is a circuit satisfying~$P$ containing every edge,
that is, an Euler circuit.

Conversely, suppose that~$G$ contains an Euler circuit satisfying~$P$.
Fix such a circuit.
Choose the phase of one vertex arbitrarily and increase the phase of the next
vertex by~$\pi/2$ at each step of the Euler circuit.
Since the circuit visits every vertex and satisfies~$P$,
this process defines a phase~$\theta_k$ at each vertex~$k$ in a consistent way.
By Lemma~\ref{lem:criterion} it follows that~$\theta$ is a completely degenerate equilibrium.
In particular, this proves that~$G$ admits completely degenerate equilibria.
\end{proof}

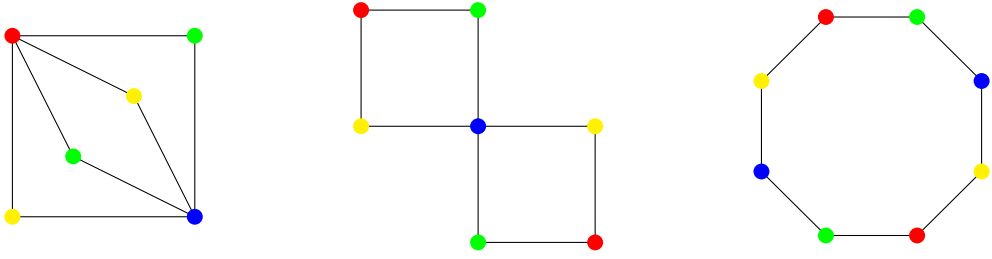
\begin{figure}
\begin{tikzcd}
\begin{tikzpicture} [scale=0.8]
\draw (0,-1.5) -- (3,-1.5);
\draw (0,1.5) -- (3,1.5);
\draw (0, -1.5) -- (0, +1.5);
\draw (3, -1.5) -- (3, +1.5);
\draw (3,-1.5) -- (1, -0.5);
\draw (3,-1.5) -- (2, 0.5);
\draw (0,1.5) -- (1, -0.5);
\draw (0,1.5) -- (2, 0.5);
\fill [blue] (3, -1.5) circle (0.1312);
\fill [green] (3, 1.5) circle (0.1312);
\fill [yellow] (0, -1.5) circle (0.1312);
\fill [red] (0, 1.5) circle (0.1312);
\fill [yellow] (2, 0.5) circle (0.1312);
\fill [green] (1, -0.5) circle (0.1312);
\end{tikzpicture}
&
\begin{tikzpicture} [scale=0.7]
\draw (0,0) -- (2.2,0);
\draw (0,0) -- (-2.2,0);
\draw (0,0) -- (0,2.2);
\draw (0,0) -- (0,-2.2);
\draw (-2.2,0) -- (-2.2,2.2);
\draw (2.2,0) -- (2.2,-2.2);
\draw (0,2.2) -- (-2.2,2.2);
\draw (0,-2.2) -- (2.2,-2.2);
\fill [blue] (0,0) circle (0.15);
\fill [green] (0,2.2) circle (0.15);
\fill [green] (0,-2.2) circle (0.15);
\fill [yellow] (2.2, 0) circle (0.15);
\fill [red] (2.2, -2.2) circle (0.15);
\fill [yellow] (-2.2, 0) circle (0.15);
\fill [red] (-2.2, 2.2) circle (0.15);
\end{tikzpicture}
&
\begin{tikzpicture} [scale=0.6]
\draw (2.414, 1) -- (1, 2.414);
\draw (1, 2.414) -- (-1, 2.414);
\draw (-2.414, 1) -- (-1, 2.414);
\draw (-2.414, 1) -- (-2.414, -1);
\draw (-2.414, -1) -- (-1, -2.414);
\draw (1, -2.414) -- (-1, -2.414);
\draw (2.414, -1) -- (1, -2.414);
\draw (2.414, 1) -- (2.414, -1);
\fill [green] (1, 2.414) circle (0.175);
\fill [red] (1, -2.414) circle (0.175);
\fill [red] (-1, 2.414) circle (0.175);
\fill [green] (-1, -2.414) circle (0.175);
\fill [blue] (2.414, 1) circle (0.175);
\fill [yellow] (-2.414, 1) circle (0.175);
\fill [yellow] (2.414, -1) circle (0.175);
\fill [blue] (-2.414, -1) circle (0.175);
\end{tikzpicture}
\end{tikzcd}
\caption{
Blue, green, red and yellow denote~$0$, $\pi/2$, $\pi$ and~$3\pi/2$ respectively.}
\label{fig:quotient}
\end{figure}

\begin{example}
The cycle graph with $4$~vertices admits completely degenerate equilibria.
An explicit example is given in Figure~\ref{fig:hypercube}.
This graph admits two completely degenerate equilibria
up to phase-shift symmetry, one for every orientation of the cycle.
\end{example}

\begin{example}
The graph formed by the vertices and edges of the $4$-dimensional hypercube
admits completely degenerate equilibria. An explicit
example is given in Figure~\ref{fig:hypercube}. More generally,
any even-dimensional hypercube admits completely degenerate equilibria.
\end{example}

There are infinitely many graphs admitting completely degenerate equilibria:
\begin{proposition} \label{prop:infinitely_many}
For every~$N\geq 6$ there is a connected graph on~$N$ vertices admitting completely degenerate
equilibria.
\end{proposition}
\begin{proof}
Explicit examples with~$N=6,7,8$ vertices are given in Figure~\ref{fig:quotient}.
Given a completely degenerate equilibrium on a connected graph with~$N$ vertices,
we can obtain a completely degenerate equilibrium on a connected graph with~$N+3$
vertices by glueing a $4$-cycle to the original graph. For example, the graph with
$7$~vertices of Figure~\ref{fig:quotient} is obtained by glueing two $4$-cycles together.
\end{proof}

Although infinite in number, graphs admitting completely degenerate equilibria are
asymptotically rare:

\begin{proposition}
The probability that a graph chosen uniformly at random
among the graphs with~$N$ vertices admits completely degenerate equilibria
goes to~$0$ as~$N$ goes to infinity.
\end{proposition}
\begin{proof}
By Lemma~\ref{lem:criterion} a graph admitting completely degenerate equilibria
is triangle-free. Let us prove that triangle-free graphs
have asymptotic probability~$0$.
Let~$G$ be a graph on~$N$ vertices.
Partition the vertices into~$\lfloor N/3 \rfloor$ subsets of size~$3$, and possibly
a subset of smaller size.
There are~$8$ distinct graphs on~$3$ vertices, and~$7$ of these are not triangles.
Therefore, the probability that~$G$ is triangle-free is at most
$
	(7/8)^{\lfloor N/3 \rfloor}.
$
In particular, the probability goes to~$0$ as~$N$ goes to infinity.
\end{proof}

\begin{remark}
From Lemma~\ref{lem:criterion} it follows that if~$G$ admits completely degenerate equilibria
then~$G$ contains no cycles of odd length. In particular~$G$ is bipartite.
That is, bipartiteness is necessary for the existence
completely degenerate equilibria in the case of zero intrinsic frequencies.
Theorem~\ref{thm:1} implies that it is not sufficient.
In a later section (Theorem~\ref{thm:bipartite}) we will show that bipartiteness is
necessary and sufficient for the existence of completely degenerate equilibria for \emph{some}
(not necessarily all zero) intrinsic frequencies.
\end{remark}

\subsection{Stability}
Stability of completely degenerate equilibria cannot be determined by linearization.
The goal of this section is proving that completely degenerate equilibria are never Lyapunov stable.
In particular, they are never asymptotically stable.

It is well known~\cite{Ling2019, Wiley2006} that~\eqref{eq:main} is a gradient system with respect
to the~\emph{energy function}
\begin{equation} \label{eq:energy}
	E(\theta) = \sum_{jk \in e(G)} a_{jk} (1 - \cos(\theta_j - \theta_k)).
\end{equation}
Here~$e(G)$ denotes the set of edges of~$G$ and each edge~$jk$ is counted exactly once in the sum.
For every point~$\theta\in \T^N$ the identity~$F(\theta) = -DE(\theta)$ holds.
Intuitively, this means that trajectories always evolve in the direction
in which~$E$ decreases maximally.

Intuitively, the energy decreases over time until a stationary point (or equilibrium) is reached.
The connection between stability of and minimality is, however, subtle.
There are smooth energy functions with
Lyapunov stable equilibria that
are not local minimizers
and local minimizers that are not Lyapunov stable~\cite{Absil2006}.
However, if the energy function is real-analytic, then the Lyapunov stable
equilibria are exactly the local minimizers of the energy~\cite{Absil2006}.
Since~\eqref{eq:energy} is real-analytic, this result applies to our case.

Since distinct connected components have independent dynamics,
it is enough to consider connected graphs.
Moreover, we assume that the graph has at least one edge.
Notice that if~$G$ has no edges then there is no dynamics and every point is a
(completely degenerate) Lyapunov stable equilibrium.

\begin{theorem} \label{thm:2}
For every connected graph with at least one edge,
completely degenerate equilibria are not Lyapunov stable.
\end{theorem}
\begin{proof}
Let~$\theta\in \T^N$ be a completely degenerate equilibrium.
We will show that~$\theta$ is a saddle point of the energy function~\eqref{eq:energy}, that is,
a point that is neither a local minimizer nor a local maximizer.
It follows that~$\theta$ is not Lyapunov-stable.

Fix an Euler circuit as in Theorem~\ref{thm:1}
and let~$j,k$ be two consecutive vertices in the circuit.
We have~$\theta_k = \theta_j + \pi/2$. For every~$x\in \R$
let~$\theta^x \in \T^N$ be defined as
\begin{equation*}
\begin{dcases*}
\theta^x_j = \theta_j + x, \\
\theta^x_k = \theta_k -x, \\
\theta^x_h = \theta_h, & $h\notin \{j,k\}$.
\end{dcases*}
\end{equation*}
Our goal is proving that~$E(\theta) - E(\theta^x)$ changes sign in every neighborhood of~$x=0$.
By~\eqref{eq:cosine_condition} and~\eqref{eq:energy} it follows that
\[
	E(\theta) - E(\theta^x) = \sum_{pq \in e(G)} \cos(\theta^x_p - \theta^x_q).
\]
Let us compute this sum by following the circuit.
The edge entering~$j$ before~$jk$, the edge~$jk$,
and the edge leaving~$k$ after~$jk$ amount to
\begin{align*}
	\cos(\pi/2 + x) + \cos(\pi/2 -2x) + \cos(\pi/2 - x) = \sin(2x) - 2\sin(x).
\end{align*}
We claim the other terms in the sum are either zero or they cancel out.
If an edge is neither adjacent to~$j$ nor~$k$,
then~$\theta^x_p = \theta_p$ and~$\theta^x_q = \theta_q$ and
therefore~$\cos(\theta^x_p - \theta^x_q) = 0$.
Every visit to~$j$ that is not the one preceding~$jk$ amounts to
\[
	\cos(\pi/2 + x) + \cos(\pi/2 - x) = - \sin(x) + \sin(x) = 0.
\]
Similarly any other visits to~$k$ that is not the one involving~$jk$ amounts to~$0$.
Therefore
\[
	E(\theta) - E(\theta^x) = \sin(2x) - 2\sin(x).
\]
Since this function changes sign in every neighborhood of~$x=0$,
the point~$\theta$ is a saddle of the energy.
\end{proof}

\section{Non-Zero Intrinsic Frequencies} \label{sec:the_end}
\subsection{Classification}
In the previous section we assumed the intrinsic frequencies to be all equal to~$0$.
Let us return to the general case
\begin{equation} \label{eq:non-identical}
	\dot \theta_k = \omega_k + K \sum_{j=1}^N a_{jk} \sin(\theta_j-\theta_k).
\end{equation}

Notice that linearizing~\eqref{eq:non-identical} gives the same Jacobian matrix
as the zero-frequency case~\eqref{eq:DF}.
It follows that a point~$\theta\in \T^N$ is a completely degenerate equilibrium
of~\eqref{eq:non-identical}
if and only if
\begin{align}
& \sum_{j} a_{jk}
	\sin(\theta_j - \theta_k) = -\frac{\omega_k}{K}
	\label{eq:sine_condition_non-identical} \\
\intertext{for every vertex~$k$ and}
&\cos(\theta_j - \theta_k) = 0
	\label{eq:cosine_condition_non-identical}
\end{align}
for every edge~$jk$. In particular we obtain:

\begin{proposition} \label{prop:generic_frequencies}
For a generic choice of intrinsic frequencies~$\omega_1,\ldots,\omega_N$
there are no graphs admitting completely degenerate equilibria.
\end{proposition}
\begin{proof}
If
equation~\eqref{eq:sine_condition_non-identical} and
equation~\eqref{eq:cosine_condition_non-identical} hold
then~$\omega_k/K$ is an integer for every~$k$.
Therefore, if there is a graph admitting completely degenerate equilibria
then the intrinsic frequencies are all elements of the lattice~$K\Z$.
\end{proof}

We showed (Proposition~\ref{prop:generic_frequencies})
that for generic intrinsic frequencies there are no graphs admitting
completely degenerate equilibria.
On the other hand, we know that for zero intrinsic frequencies there are infinitely many graphs
admitting completely degenerate equilibria (Proposition~\ref{prop:infinitely_many}).
This suggests asking which graphs admit completely degenerate equilibria for
\emph{some} intrinsic frequencies:

\begin{theorem} \label{thm:bipartite}
For every graph~$G$ the following facts are equivalent:
\begin{enumerate} [(i)]
\item There are coupling strength~$K$ and intrinsic frequencies~$\omega_1,\ldots,\omega_N$
such that~$G$ admits completely degenerate equilibria;
\item $G$ is bipartite.
\end{enumerate}
\end{theorem}
\begin{proof}
Let~$\theta$ be a completely degenerate equilibrium.
By~\eqref{eq:cosine_condition_non-identical} the phase differences
of adjacent vertices are~$\pm \pi/2$. In particular~$G$ contains
no cycle of odd length. This is equivalent
to bipartiteness.

Conversely, suppose that~$G$ is bipartite.
Let the phases of one part be all equal to~$0$
and of the other part all equal to~$\pi/2$.
Then~\eqref{eq:cosine_condition_non-identical} is satisfied.
Now choose any~$K>0$ and for every vertex~$k$ define
\[
	\omega_k = -K \sum_{j} a_{jk} \sin(\theta_j - \theta_k).
\]
Then~\eqref{eq:sine_condition_non-identical} is also satisfied
and~$\theta$ is a completely degenerate equilibrium.
\end{proof}

\subsection{Stability}
As we will see, the argument used to determine instability in the case of zero
intrinsic frequencies can be adapted to the general case.
There is, however, a technical difference:
while~\eqref{eq:main} is a gradient system on~$\T^N$,
in general~\eqref{eq:non-identical} is only a gradient system locally, in a neighborhood
of the equilibrium.
This is explained in more details in the following proof.

\begin{theorem} \label{thm:2_non-identical}
For every connected graph with at least one edge, every coupling strength~$K>1$
and every intrinsic frequencies~$\omega_1,\ldots,\omega_N\in \R$,
the completely degenerate equilibria are not Lyapunov stable.
\end{theorem}
\begin{proof}
Up to rescaling time we can suppose~$K=1$.
Let~$\theta \in \T^N$ be a completely degenerate equilibrium.
By~\eqref{eq:sine_condition_non-identical} and~\eqref{eq:cosine_condition_non-identical}
the neighbors of a vertex~$k$ have phases $\theta_k + \pi/2$~or~$\theta_k - \pi/2$.
If, for every vertex~$k$, half of its neighbors have phase~$\theta_k + \pi/2$ and
the other half~$\theta_k - \pi/2$, then from equation~\eqref{eq:sine_condition_non-identical}
it follows that~$\omega_k=0$ for every~$k$ and Theorem~\ref{thm:2} applies.

Otherwise, there is a vertex~$k$ such that~$k$ has~$d^+$ neighbors with phase~$\theta_k + \pi/2$
and~$d^-$ neighbors with phase~$\theta_k - \pi/2$, where $d^+\neq d^-$.
For every~$x\in \R$ let~$\theta^x \in \T^N$ be defined as
\begin{equation*}
\begin{dcases*}
\theta^x_k = \theta_k +x, \\
\theta^x_j = \theta_j, & $j \neq k$.
\end{dcases*}
\end{equation*}
In a neighborhood of~$\theta$ the system~\eqref{eq:non-identical} is gradient with energy function
\[
	E(\theta) = -\sum_k \omega_x\theta_k
		+ \sum_{jk \in e(G)} (1 - \cos(\theta_j - \theta_k)).
\]
Notice that the terms~$\omega_x\theta_k$ are only well-defined locally since
there are no injective continuous maps from~$\T$ to~$\R$.
This is not an obstacle for the proof:
all we need is the function~$x\mapsto E(\theta^x)$ to be defined
in some neighborhood of~$x=0$. We have
\begin{align*}
	E(\theta) - E(\theta^x) &= \omega_k x
		+ \sum_j a_{jk} (\cos(\theta_j - \theta_k - x) - \cos(\theta_j - \theta_k)) \\
		&= \omega_k x + d^+ \cos(\pi/2-x) + d^- \cos(-\pi/2-x) \\
		&= \omega_k x + (d^+-d^-) \sin(x).
\end{align*}
From
equation~\eqref{eq:sine_condition_non-identical} and
equation~\eqref{eq:cosine_condition_non-identical} it follows that~$d^+-d^- = -\omega_k$.
Therefore
\[
	E(\theta) - E(\theta^x) = (d^+-d^-)(-x + \sin(x)).
\]
This function changes sign in every neighborhood of~$x=0$.
We conclude that~$\theta$ is not Lyapunov stable.
\end{proof}

As a corollary we obtain the following general result:

\begin{corollary} \label{cor:strictly_neg_eig}
For every graph with at least one edge,
every coupling strength and every intrinsic frequencies,
the Lyapunov stable equilibria of~\eqref{eq:non-identical}
have at least one strictly negative eigenvalue.
\end{corollary}
\begin{proof}
Without loss of generality, it is enough to consider connected graphs.
Let~$\theta$ be a Lyapunov stable equilibria.
By the Center Manifold Theorem the Jacobian eigenvalues at~$\theta$
are negative or zero.
By Theorem~\ref{thm:2_non-identical} they are not all zero.
\end{proof}

Corollary~\ref{cor:strictly_neg_eig} shows that the presence of strictly negative eigenvalues,
although not sufficient, is necessary for stability.
Therefore, albeit linear stability analysis is not enough to determine
the set of stable equilibria, it helps by restricting the search.

\begin{remark}
The proof of Theorem~\ref{thm:2_non-identical} actually shows something more:
completely degenerate equilibria are saddles of the energy.
In particular, it follows that they are unstable in both forward and backward
time, or if~$K$ is chosen to be negative.
\end{remark}

\bibliographystyle{siam}
\bibliography{refs}

\end{document}